\documentclass{amsart}
\usepackage{amssymb,amsthm, amsmath, amsfonts, amsbsy}
\usepackage{graphics}
\usepackage{hyperref}
\usepackage[all]{xy}
\usepackage{enumerate}
\usepackage[mathscr]{eucal}
\usepackage{bbm}

\theoremstyle{plain}
\newtheorem{theorem}{Theorem}[section]
\newtheorem{lemma}[theorem]{Lemma}
\newtheorem{proposition}[theorem]{Proposition}

\newtheorem{corollary}[theorem]{Corollary}
\numberwithin{equation}{section}

\theoremstyle{definition}

\newtheorem{definition}[theorem]{Definition}
\newtheorem{example}[theorem]{Example}
\newtheorem{remark}[theorem]{Remark}

\DeclareMathOperator{\reg}{reg}
\DeclareMathOperator{\Mod}{-Mod}

\DeclareMathOperator{\Hom}{Hom}
\DeclareMathOperator{\Tor}{Tor}
\DeclareMathOperator{\hd}{hd}
\DeclareMathOperator{\gd}{gd}
\DeclareMathOperator{\pd}{pd}
\DeclareMathOperator{\td}{td}
\DeclareMathOperator{\fdim}{findim}

\newcommand{\C}{{\mathscr{C}}}
\newcommand{\N}{{\mathbb{N}}}
\newcommand{\mJ}{{\mathfrak{J}}}
\newcommand{\Nm}{{\mathbb{N}^m}}
\newcommand{\bfl}{{\mathbf{l}}}
\newcommand{\bft}{{\mathbf{t}}}
\newcommand{\bfn}{{\mathbf{n}}}
\newcommand{\bff}{{\mathbf{f}}}
\newcommand{\bfg}{{\mathbf{g}}}
\newcommand{\bfh}{{\mathbf{h}}}
\newcommand{\bfo}{{\mathbf{1}}}
\newcommand{\bfs}{{\boldsymbol{\Sigma}}}
\newcommand{\bfd}{{\mathbf{D}}}
\newcommand{\bfk}{{\mathbf{K}}}
\newcommand{\FI}{{\mathrm{FI}}}
\newcommand{\FIM}{{\mathrm{FIM}}}
\newcommand{\OI}{{\mathrm{OI}}}
\newcommand{\mm}{{\mathfrak{m}}}

\title{$\FI^m$-modules over Noetherian rings}
\author{Liping Li}
\address{Key Laboratory of Performance Computing and Stochastic Information Processing (Ministry of Education), College of Mathematics and Computer Science, Hunan Normal University; Changsha, Hunan 410081, China.}
\email{lipingli@hunnu.edu.cn}
\thanks{The first author was supported by the National Natural Science Foundation of China 11541002, the Construct Program of the Key Discipline in Hunan Province, and the Start-Up Funds of Hunan Normal University 830122-0037. The second author was supported by China NSF 11601452.}

\author{Nina Yu}
\address{School of Mathematical Sciences, Xiamen University, Xiamen, Fujian, 361005, China.}
\email{ninayu@xmu.edu.cn.}

\begin{document}

\begin{abstract}
In this paper we study representation theory of the category $\FI^m$ introduced in \cite{Gad1, Gad2} which is a product of copies of the category $\FI$, and show that quite a few interesting representational and homological properties of $\FI$ can be generalized to $\FI^m$ in a natural way. In particular, we prove the representation stability property of finitely generated $\FI^m$-modules over fields of characteristic 0.
\end{abstract}

\maketitle

\section{Introduction}

\subsection{Motivation}

Representation theory of the category $\FI$ and its plentiful applications in various fields such as algebra, algebraic topology, algebraic geometry, number theory, rooted from the work of Church, Ellenberg, Farb, and Nagpal in \cite{CEF, CEFN}, have been actively studied. With their initial contributions and the followed explorations described in a series of papers \cite{CE, Gan1, Gan2, GL1, LGO, L1, L2, LR, LY, R1, R2, SS1, SS2}, people now have a satisfactory understanding on representational and homological properties of $\FI$-modules and their relationships with representation stability properties. Furthermore, quite a few combinatorial categories sharing similar structures as $\FI$, including $\FI_G$, $\FI_d$, $\FI^m$, $\FIM$ have been introduced and studied with viewpoints from representation theory, commutative algebra, and combinatorics; see \cite{Gad1, Gad2, SS2, W}.

The main goal of this paper is to investigate representations of $\FI^m$, a generalization of $\FI$ introduced by Gadish in \cite{Gad1, Gad2}. Intuitively, $\FI^m$ is a product of $m$ copies of $\FI$, so it is reasonable to expect that methods to study $\FI$-modules and their outcomes can   extend to $\FI^m$ in a natural way. As the reader will see, this is indeed the case, although the techniques become more complicated and subtle. In particular, we show that $\FI^m$ is locally Noetherian over any commutative Noetherian ring, and homological properties of \emph{relative projective $\FI$-modules} (or \emph{$\sharp$-filtered modules} in literature) still hold for $\FI^m$. As an important application, we show that finitely generated $\FI^m$-modules over fields of characteristic 0 have representation stability.

\subsection{Notation}

Throughout this paper let $m$ be a positive integer. By definition (see \cite{Gad1, Gad2}), objects of the category $\FI^m$ are $m$-tuples of finite sets $\mathbf{T} = (T_1, \, \ldots, \, T_m)$, and morphisms from an object $\mathbf{T}$ to another object $\mathbf{T'}$ are maps $\bff = (f_1, \, \ldots, \, f_m)$ such that each $f_i: T_i \to T_i'$ is injective. It has a skeletal full subcategory $\C$, whose objects are $\bfn = ([n_1], \, \ldots, \, [n_m])$ where $[n_i] = \{1, \, 2, \, \ldots, \, n_i\}$ and by convention $[0] = \emptyset$. When $m = 1$, $\FI^m$ coincides with $\FI$.

Note that objects of $\C$ form a \textit{ranked poset} isomorphic to $\Nm$, where $\mathbb{N}$ is the set of all nonnegative integers. That is, $\bft \leqslant \bfn$ if and only if $t_i \leqslant n_i$ for all $i \in [m]$. To simplify the notation, we identify objects in $\C$ with elements in $\Nm$ and hope that this simplification would not cause too much confusion to the reader. The \textit{degree} of an object $\bfn = ([n_1], \, \ldots, \, [n_m])$ is defined to be $\deg (\bfn) = n_1 + \ldots + n_m$. The degree of a morphism $\bff \in \C (\bfn, \bft)$ is defined to be $\deg(\bft) - \deg(\bfn)$. With this degree function the category $\FI^m$ becomes a \emph{graded category}.

Let $k$ be a unital commutative ring. A \textit{representation} of $\C$, or a $\C$-\emph{module}, is a covariant functor $V$ from $\C$ to $k \Mod$, the category of $k$-modules. It is well known that $\C \Mod$ is an abelian category. Moreover, it has enough projectives. In particular, for $\bfn \in \Nm$, the $k$-linearization of the representable functor $\C(\bfn, -)$ is projective. We call it a \textit{free module}, and denote it $M(\bfn)$. The value of a representation $V$ on an object $\bfn$ is denoted by $V_{\bfn}$.

A representation $V$ of $\C$ is said to be \emph{finitely generated} if there exists a finite subset $S$ of $V$ such that any submodule containing $S$ coincides with $V$; or equivalently, there exists a surjective homomorphism
\begin{equation*}
\bigoplus _{\bfn \in \Nm} M(\bfn) ^{\oplus a_{\bfn}} \to V
\end{equation*}
such that $\sum _{\bfn \in \Nm} a_{\bfn} < \infty$. It is said to be \emph{generated in degrees $\leqslant N$} if in the above surjection one can let $a_{\bfn} = 0$ for all $\bfn$ with $\deg(\bfn) > N$. It is easy to see that $V$ is finitely generated if and only if it is generated in degrees $\leqslant N$ for a certain $N \in \N$ and the values $V_{\bfn}$ with $\deg(\bfn) \leqslant N$ are finitely generated $k$-modules.

The category $\C$ has $m$ distinct \emph{self-embedding functors} $\iota_1, \, \ldots, \, \iota_m$ of degree 1. That is, for each $i \in [m]$, $\iota_i$ is a faithful functor $\C \to \C$ such that
\begin{equation*}
\bfn = (n_1, \, \ldots, \, n_i, \, \ldots, \, n_m) \mapsto (n_1, \, \ldots, \, n_i + 1, \, \ldots, \, n_m) = \bfn + \bfo_i
\end{equation*}
where $\bfo_i = (0, \ldots, 0, \, 1, \, 0, \, \ldots, \, 0)$ and $1$ is on the $i$-th position. They induce $m$ distinct \emph{shift functors} $\Sigma_i: \C \Mod \to \C \Mod$. Let $\bfs$ be the direct sum of these shift functors $\Sigma_i$, $i \in [m]$. There exist natural maps $V \to \Sigma_i V$, and hence a natural map $V^{\oplus m} \to \bfs V$. The \emph{derivative functor} $D$ is defined to be the cokernel functor induced by these natural maps $V^{\oplus m} \to \bfs V$. Therefore, for each $\C$-module $V$, we have an exact sequence $0 \to KV \to V^{\oplus m} \to \bfs V \to DV \to 0$.

\subsection{Noetherianity}

By the ranked poset structure on the set of objects, the category $\C$ has a two-sided ideal of morphisms
\begin{equation*}
\bigsqcup _{\mathbf{0} \leqslant \mathbf{i} < \mathbf{j}} \C(\mathbf{i}, \mathbf{j}).
\end{equation*}
Denote the $k$-linearization of this ideal by $\mm$, which is a two-sided ideal of the \emph{category algebra} $k \C$ (for a definition, see \cite{Webb}). Given a $\C$-module $V$ and a nonnegative integer $s$, the $s$-th \emph{homology group} of $V$ is defined to be
\begin{equation*}
H_s (V) = \Tor _s^{k \C} (k \C / \mm, V).
\end{equation*}
This is a $k \C$-module since $\mm$ is a $(k \C, k \C)$-bimodule. Accordingly, the $s$-th homological degree is
\begin{equation*}
\hd_s(V) = \sup \{ \deg(\bfn) \mid (H_s(V))_{\bfn} \neq 0 \}.
\end{equation*}
By convention, we set $\hd_s(V)$ to be -1 whenever the above set is empty. We call the zeroth homological degree \emph{generating degree}, and denote it by $\gd(V)$.

When $k$ is a commutative Noetherian ring, it is easy to see that a \emph{locally finite} (see Definition \ref{locally finite}) $\C$-module $V$ is finitely presented if and only if $\hd_1(V) < \infty$. Using this homological characterization, we can prove:

\begin{theorem} \label{noetherianity}
The category $\C$ is locally Noetherian over any commutative Noetherian ring $k$. That is, submodules of finitely generated $\C$-modules are still finitely generated.
\end{theorem}

\begin{remark} \normalfont
For $m = 1$, this result was proved in \cite[Theorem A]{CEFN}. For arbitrary $m \geqslant 1$, Gadish proved in \cite{Gad2} this theorem for fields of characteristic 0.
\end{remark}

\subsection{Relative projective modules}

For each object $\bfn$, its endomorphisms form a group $S_{\bfn}$ isomorphic to $S_{n_1} \times \ldots \times S_{n_m}$, where $S_{n_i}$ is a symmetric group on $n_i$ letters. Since the group algebra $k S_{\bfn}$ is a subalgebra of the category algebra $k \C$ and $M(\bfn)$ is a $(k \C, k S_{\bfn})$-bimodule, given a $k S_{\bfn}$-module $U$, it induces a $\C$-module $M(\bfn) \otimes _{k S_{\bfn}} U$. We call these modules \emph{basic relative projective modules}. A $\C$-module $V$ is said to be \emph{relative projective} if it has a filtration
\begin{equation*}
0 = V^{-1} \subseteq V^0 \subseteq \ldots \subseteq V^s = V
\end{equation*}
such that $V^{i+1}/V^i$ is isomorphic to a basic relative projective module for $-1 \leqslant i \leqslant s-1$. They are also called \emph{$\sharp$-filtered modules} or just \emph{filtered modules} in literature such as \cite{L2, LR, LY, N}. As the reader can see, they generalize projective modules, and have similar homological properties.

\begin{theorem} \label{homological characterizations of relative projective modules}
Let $V$ be $\C$-module over a commutative ring $k$, and suppose that $\gd(V) < \infty$. Then the following statements are equivalent:
\begin{enumerate}
\item $V$ is relative projective;
\item $H_s(V) = 0$ for all $s \geqslant 1$;
\item $H_1(V) = 0$;
\item $H_s(V) = 0$ for some $s \geqslant 1$.
\end{enumerate}
\end{theorem}

\begin{remark} \normalfont
The case of $m = 1$ was independently proved by the authors in \cite{LY} and Ramos in \cite{R1}.
\end{remark}

\subsection{Torsion theory}

Let $V$ be a $\C$-module. An element $v \in V_{\bfn}$ is said to be \emph{torsion} if there exist another object $\bft$ and an injection $\bff: \bfn \to \bft$ such that $\bff \cdot v = 0$. Torsion elements in $V$ generate a submodule of $V$, denoted by $V_T$. If $V_T = 0$, we say that $V$ is \emph{torsion free}; if $V_T = V$, we call $V$ a \emph{torsion module}. In general, there is a short exact sequence $0 \to V_T \to V \to V_F \to 0$, where $V_F$ and $V_T$ are the \emph{torsion free part} and the \emph{torsion part} of $V$ respectively. The \emph{torsion degree} $\td(V)$ of $V$ is defined in Definition \ref{torsion degree}.

The following theorem generalizes the corresponding result of finitely generated $\FI$-modules over commutative Noetherian rings.

\begin{theorem} \label{relative projective complexes}
Let $V$ be finitely generated $\C$-module over a commutative Noetherian ring $k$. There exists a complex
\begin{equation*}
F^{\bullet}: \quad 0 \to V \to F^0 \to F^1 \to \ldots \to F^l \to 0
\end{equation*}
such that the following statements hold:
\begin{enumerate}
\item each $F^j$ is a relative projective module with $\gd(F^j) \leqslant \gd(V) - j$;
\item $l \leqslant \gd(V)$,
\item all homology groups $H^j (F^{\bullet})$ of this complex are finitely generated torsion modules.
\end{enumerate}

Consequently, $\Sigma_1^{n_1} \ldots \Sigma_m^{n_m} V$ is a relative projective module if $n_i \geqslant \td_i(H^j (F^{\bullet})) + 1$ for all $0 \leqslant j \leqslant l$ and $i \in [m]$.
\end{theorem}

By the theorem, we define $N_i(V) = \max \{ \td_i(H^j (F^{\bullet})) \mid j \geqslant 0 \}$ for $i \in [m]$, which are finite numbers independent of the choice of a particular complex.

\begin{remark} \normalfont
For $m = 1$, this result was verified by Nagpal in \cite[Theorem A]{N}. Other proofs with explicit upper bounds were given by the authors in \cite{L2, LY} and Ramos in \cite{R1}. In this paper we will modify the techniques in \cite{L2, LY} to prove the above theorem.

When $m = 1$, this complex has played a vital role in establishing upper bounds of homological invariants and developing a local cohomology theory for $\FI$-modules; see \cite{L2, LR, LY}.
\end{remark}

\subsection{Representation stability}

The above results can be used to prove certain representation stability patterns of finitely generated $\C$-modules over fields. Let us recall some notation. A \emph{partition} of a nonnegative integer $n$ is a sequence $\lambda = (\lambda_1 \geqslant \lambda_2 \geqslant \ldots \geqslant \lambda_l)$ such that $|\lambda| = \lambda_1 + \ldots + \lambda_l = n$. For $t \geqslant n + \lambda_1$, we define the \emph{padded partition} to be
\begin{equation*}
\lambda(t) = (t -n \geqslant \lambda_1 \geqslant \ldots \geqslant \lambda_l)
\end{equation*}
which is a partition of $m$.

For $\bfn = (n_1, \, \ldots, \, n_m) \in \Nm$, an \emph{$m$-fold partition} $\boldsymbol{\lambda} = (\lambda^1, \, \ldots, \, \lambda^m)$ such that $\lambda^i$ is a partition of $n_i$ for $i \in [m]$. Let $\boldsymbol{\lambda}_1 = (\lambda^1_1, \, \ldots, \, \lambda^m_1) \in \Nm$. For $\bft \geqslant \bfn + \boldsymbol{\lambda}_1$, one defines an \emph{$m$-fold padded partition}
\begin{equation*}
\boldsymbol{\lambda} (\bft) = (\lambda^1 (t_1), \, \ldots, \, \lambda^m (t_m)).
\end{equation*}

When $k$ is a field of characteristic 0, for an object $\bfn$ in $\C$ and its endomorphism group $S_{\bfn} \cong S_{n_1} \times \ldots \times S_{n_m}$, from group representation theory we know that simple $k S_{\bfn}$-modules are parameterized by $m$-fold partitions of $\bfn$. Moreover, when $\bft \geqslant \bfn + \boldsymbol{\lambda}_1$, $m$-fold padded partitions provide a uniform way to describe irreducible modules of distinct $S_{\bft}$. Denote by $L_{\boldsymbol{\lambda} (\bft)}$ the irreducible $k S_{\bft}$-module parameterized by the $m$-fold partition $\boldsymbol{\lambda} (\bft)$.

Motivated by \cite{Gad2}, we define representation stability of $\FI^m$-modules as follows:

\begin{definition} \label{representation stability}
Let $k$ be a field of characteristic 0 and $V$ be a finitely generated $\C$-module. We say that $V$ has representation stability if for all morphisms $\bff: \bfn \to \bft$, when $n_i \gg 0$ for $i \in [m]$, one has:
\begin{itemize}
\item the linear map $V(\bff): V_{\bfn} \to V_{\bft}$ is injective;
\item the image of $V(\bff)$ generates $V_{\bft}$ as a $k S_{\bft}$-module;
\item there exist finitely many $m$-fold partitions $\boldsymbol{\lambda}^{(1)}, \, \ldots, \boldsymbol{\lambda}^{(k)}$ such that
\begin{equation*}
V_{\bfn} = \bigoplus _{i \in [k]} c_i L_{\boldsymbol{\lambda}^{(i)}} (\bfn),
\end{equation*}
where the multiplicity $c_i$ is independent of $\bfn$.
\end{itemize}
\end{definition}

Our next theorem generalizes representation stability of $\FI$-modules proved by Church, Ellenberg, and Farb in \cite{CEF}.

\begin{theorem} \label{representation stability property}
Let $V$ be a $\C$-module over a field of characteristic 0. Then $V$ is finitely generated if and only if it has representation stability. Moreover, in that case the numbers $n_i$, $i \in [m]$, in the above definition only need to satisfy $n_i \geqslant \max \{2\gd(V), \, N_i(V) + 1\}$.
\end{theorem}

\begin{remark} \normalfont
Gadish showed the representation stability for finitely generated projective $\C$-modules in \cite[Theorem 6.13]{Gad2}. \footnote{Note that \emph{free $\C$-modules} have different meanings in this paper and in \cite{Gad2}. Free $\C$-modules defined in \cite[Definition 1.8]{Gad2} are direct sums of basic relative projective modules in our sense, and when $k$ is a field of characteristic 0, they coincide with projective $\C$-modules.} The conclusion of the above theorem follows immediately from his result as well as Theorem \ref{relative projective complexes}.
\end{remark}

Another important asymptotic behavior is the polynomial growth property. We have:

\begin{theorem} \label{polynomial growth}
Let $V$ be a finitely generated $\C$-module over a field. Then there exist $m$ polynomials $P_i \in \mathbb{Q}[X]$, $i \in [m]$, with degrees not exceeding $\gd(V)$ such that for objects $\bfn$ satisfying $n_i \geqslant \max \{\gd(V), \, N_i(V) + 1 \}$ for all $i \in [m]$, one has
\begin{equation*}
\dim_k V_{\bfn} = \prod _{i \in [m]} P_i(n_i).
\end{equation*}
\end{theorem}

\subsection{Projective dimensions}

Homological properties of relative projective modules can also be used to classify finitely generated $\C$-modules whose projective dimension is finite, extending \cite[Theorem 1.5]{LY}.

\begin{theorem} \label{projective dimension}
Let $k$ be a commutative Noetherian ring whose finitistic dimension $\fdim k$ is finite, and let $V$ be a finitely generated $\C$-module. Then the projective dimension $\pd _{k \C} (V)$ is finite if and only if there exist a finite set $S$ of objects $\bfn$ and finitely generated $k S_{\bfn}$-modules $W_{\bfn}$, such that $V$ is a relative projective module whose filtration components are $M(\bfn) \otimes _{k S_{\bfn}} W_{\bfn}$ and $\pd _{k} (W_{\bfn}) < \infty$ for all $\bfn \in S$. Moreover, in that case
\begin{equation*}
\pd(V) = \max \{ \pd _{k} (W_{\bfn}) \} _{\bfn \in S} \leqslant \fdim k.
\end{equation*}
\end{theorem}

\subsection{Remarks}

The behaviors of $\FI_G$ (see \cite{LR} for a definition) are very similar to those of $\FI$, where $G$ is a finite group. One can define product categories $\FI_G^m$ and shift functors, and establish analogue versions of the main results listed above  for $\FI_G^m$-modules.

The category $\OI$ was introduced in \cite{SS2}. Its objects are finite linearly ordered sets, and morphisms are order-preserving injections. One can define the product category $\OI^m$ for $m \geqslant 1$ together with shift functors for $\OI^m$-modules. Using the techniques described in Section 3, we can prove a similar version of Theorem \ref{noetherianity} for $\OI^m$. However, since Statement (1) of Lemma \ref{basic combinatorics} fails for $\OI^m$, we cannot extend other main results of $\FI^m$ in this paper to $\OI^m$.

We would like to thank Wee Liang Gan for many valuable discussions and comments on our manuscript. After the paper was posted on arXiv, Sam told us that the Noetherianity of $\FI^m$ is an immediately corollary of \cite[Theorem 1.1.3 and Proposition 4.3.5]{SS2} since $\FI$ is known to be quasi-Gr\"{o}bner. The authors were also notified by Casto that in \cite{Casto} he had independently proved Theorem \ref{relative projective complexes} over fields of characteristic 0 and Theorem \ref{representation stability property} of this paper, and described quite a few interesting applications of these results on arithmetics and geometry. The authors thank them for the discussions, comments, and communications.

\section{Preliminary results}

Throughout this section let $k$ be a unital commutative ring, $\C$ be the skeletal subcategory of $\FI^m$ with objects parameterized by $\bfn \in \Nm$, and $V$ be a $\C$-module.

\subsection{Some combinatorics}

Recall that objects in $\C$ are elements $\bfn =(n_1, \, \ldots, \, n_m) \in \Nm$ (or more precisely, objects are parameterized by $\bfn \in \Nm$), and morphisms from object $\bfn$ to $\bft$ are maps $\bff = (f_1, \, \ldots, \, f_m)$ such that each $f_i: [n_i] \to [t_i]$ is an injection, $i \in [m]$. Define $\bfn \leqslant \bft$ if the set $\C (\bfn, \bft)$ of morphisms is nonempty. This is a well defined partial order on the set of objects in $\C$ and is compatible with the usual order on $\Nm$; that is, $\bfn \leqslant \bft$ if and only if $n_i \leqslant t_i$ for all $i \in [m]$.

We describe some combinatorial properties of $\C$, most of which can be easily verified from the corresponding results of $\FI$.

\begin{lemma} \label{basic combinatorics}
Let $\bfn$, $\bft$ and $\bfl$ be three distinct objects in $\C$. One has:
\begin{enumerate}
\item If $\bfn \leqslant \bft$, then the endomorphism group $S_{\bft}$ acts transitively on $\C (\bfn, \bft)$ from the left and the endomorphism group $S_{\bfn}$ acts freely on $\C(\bfn, \bft)$ from the right.
\item If $\bff$ and $\bfg$ are two distinct morphisms in $\C(\bfn, \bft)$ and $\bfh \in \C (\bft, \bfl)$, then $\bfh \circ \bff \neq \bfh \circ \bfg$.
\item For any sequence of objects $\bfn \leqslant \bfn_1 \leqslant \bfn_2 \leqslant \ldots \leqslant \bfn_s \leqslant \bft$, one has
\begin{equation*}
\C (\bfn_s, \bft) \circ \C (\bfn_{s-1}, \bfn_s) \circ \ldots \circ \C(\bfn_1, \bfn_2) \circ \C (\bfn, \bfn_1) = \C (\bfn, \bft).
\end{equation*}
\end{enumerate}
\end{lemma}

\begin{proof}
As claimed, these statements can be deduced from the corresponding results of $\FI$. We give a proof for the second statement as an explanation. Since $\bff \neq \bfg$, we know that there exists an $i \in [m]$ such that $f_i: [n_i] \to [t_i]$ is different from $g_i: [n_i] \to [t_i]$. Therefore, by the corresponding result of $\FI$, one knows that $h_i \circ g_i \neq h_i \circ f_i$. Consequently, $\bfh \circ \bff \neq \bfh \circ \bfg$.
\end{proof}

An important combinatorial property of $\C$ is the existence of self-embedding functors of degree 1. For $i \in [m]$, we define a functor $\iota_i: \C \to \C$ as follows. For $\bfn \in \Nm$, one has
\begin{equation*}
\bfn = ([n_1], \, \ldots, \, [n_{i-1}], \, [n_i], \, [n_{i+1}], \, \ldots, \, [n_m]) \mapsto ([n_1], \, \ldots, \, [n_{i-1}], \, [n_i + 1], \, [n_{i+1}], \, \ldots, \, [n_m]) = \bfn + \bfo_i.
\end{equation*}
For a map
\begin{equation*}
\bff = (f_1, \, \ldots, \, f_m): \bfn \to \bft,
\end{equation*}
the $j$-th component $(\iota_i(\bff))_j$ of the map
\begin{equation*}
\iota_i (\bff): \bfn + \bfo_i \to \bft + \bfo_i
\end{equation*}
coincides with $f_j$ for $j \neq i$, while the $i$-th component $(\iota_i (\bff))_i: [n_i+1] \to [t_i+1]$ is defined by
\begin{equation} \label{iota}
t \mapsto \begin{cases}
1, & 1 = t \in [n_i+1];\\
f_i(t-1) + 1, & 1 \neq t \in [n_i+1].
\end{cases}
\end{equation}
The reader can check that each $\iota_i$ is a faithful functor.

\begin{lemma} \label{self-embedding functors commute}
For $i, j \in [m]$, one has $\iota_i \circ \iota_j = \iota_j \circ \iota_i$.
\end{lemma}

\begin{proof}
It follows directly from the definition. Without loss of generality we assume that $i < j$. For an object $\bfn$ in $\C$, one has
\begin{equation*}
\iota_i (\iota_j (\bfn)) = \bfn + \bfo_j + \bfo_i = \bfn + \bfo_i + \bfo_j = \iota_j (\iota_i (\bfn)).
\end{equation*}
For a morphism $\bff: \bfn \to \bft$, both $\iota_i (\iota_j (\bff))$ and $\iota_j (\iota_i (\bff))$ equal
\begin{equation*}
(f_1, \, \ldots, \, f_{i-1}, \, (\iota_i(\bff))_i, \, f_{i+1}, \, \ldots, \, f_{j-1}, \, (\iota_j(\bff))_j, \, f_{j+1}, \, \ldots, \, f_m).
\end{equation*}
\end{proof}

For each $\iota_i$, there is a family of inclusions
\begin{equation} \label{pi}
\{ \pi_{\bfn, i}: \bfn = ([n_1], \, \ldots, \, [n_i], \, \ldots, \, [n_m]) \to ([n_1], \, \ldots, \, [n_i + 1], \, \ldots, \, [n_m]) = \bfn + \bfo_i \mid \bfn \in \Nm \}
\end{equation}
such that the $j$-th component of $\pi_{\bfn, i}$ is the identity map for $j \neq i$, and the $i$-th component of $\pi_{\bfn, i}$ maps $t \in [n_i]$ to $t+1 \in [n_i + 1]$. This collection of maps gives a natural transformation $\pi_i$ between the identity functor $\mathrm{Id}_{\C}$ and the self-embedding functor $\iota_i$.

\subsection{Shift functors}

Each self-embedding functor $\iota_i$, $i \in [m]$, induces a pull-back functor $\Sigma_i: \C \Mod \to \C \Mod$ by sending $V$ to $V \circ \iota_i$. This is an exact functor called the $i$-th \emph{shift functor}. The natural transformations $\pi_i$, $i \in [m]$, induce natural transformations $\pi_i^{\ast}$ between the identity functor on $\C \Mod$ and $\Sigma_i$. Consequently, there is a natural map $V \to \Sigma_i V$ for each $i \in [m]$.

Let $\bfs$ be the direct sum of those pull-back functors $\Sigma_i$, which is exact as well. By taking direct sum, we obtain a natural map $V^{\oplus m} \to \bfs V$. The \emph{derivative functor} $\bfd$ of $\bfs$ is defined to be the cokernel of the natural map $V^{\oplus m} \to \bfs V$. We also define $\bfk V$ to be the kernel of this natural map. With these definitions, we get an exact sequence
\begin{equation*}
0 \to \bfk V \to V^{\oplus m} \to \bfs V \to \bfd V \to 0.
\end{equation*}
Note that both $\bfk$ and $\bfd$ are also direct sums of components. Explicitly, for each $i \in [m]$, one has an exact sequence
\begin{equation*}
0 \to K_iV \to V \to \Sigma_i V \to D_i V \to 0
\end{equation*}
which is precisely the $i$-th component of the previous one.

We list certain properties of these functors.

\begin{lemma} \label{basic properties of sigma}
Let $V$ be a $\C$-module. Then one has:
\begin{enumerate}
\item $\Sigma_i M(\bfn) \cong M(\bfn) \oplus M(\bfn - \bfo_i)^{\oplus n_i}$.
\item $D_i M(\bfn) \cong M(\bfn - \bfo_i)^{\oplus n_i}$.
\item $\gd(\bfd V) \leqslant \gd(\bfs V) \leqslant \gd(V) = \gd(\bfd V) + 1$ whenever $V$ is nonzero.
\item For $i, j \in [m]$, $\Sigma_i \circ \Sigma_j = \Sigma_j \circ \Sigma_j$.
\item For $i, j \in [m]$, $\Sigma_i \circ D_j \cong D_j \circ \Sigma_i$, and in particular, $\bfs \circ \bfd \cong \bfd \circ \bfs$.
\end{enumerate}
\end{lemma}

\begin{proof}
(1): The first statement follows from the $m = 1$ case. Indeed, for an object $\bfn$, the free module $M(\bfn)$ is the external tensor product of free $\FI$-modules. Explicitly,
\begin{equation*}
M(\bfn) \cong M(n_1) \boxtimes \ldots \boxtimes M(n_m),
\end{equation*}
where $M(n_i)$ is the free $\FI$-module associated with object $[n_i]$ in $\FI$. Now we have
\begin{equation*}
\Sigma_i M(\bfn) = M(n_1) \boxtimes \ldots \boxtimes (\Sigma M(n_i)) \boxtimes \ldots \boxtimes M(n_m)
\end{equation*}
and
\begin{equation*}
\Sigma M(n_i) \cong M(n_i) \oplus M(n_i-1)^{\oplus n_i},
\end{equation*}
where $\Sigma$ is the shift functor of $\FI$-modules defined in \cite{CEF}. The conclusion follows.

(2): Immediately follows from the previous statement.

(3): Take a surjection $P \to V \to 0$ such that $P$ is a direct sum of free $\C$-modules with $\gd(P) = \gd(V)$. Applying the shift functor one gets a surjection $\bfs P \to \bfs V \to 0$. Since $\gd(\bfd P) = \gd(P)$ by the Statement (1) of this lemma, one deduces that $\gd(\bfs V) \leqslant \gd(V)$. Since $\bfd V$ is a quotient module of $\bfs V$, one also has $\gd(\bfs V) \geqslant \gd(\bfd V)$.

Now we show that $\gd(\bfd V) = \gd(\bfd(V)) + 1$ whenever $V \neq 0$, the proof of which is similar to that of \cite[Lemma 1.5]{L3} and \cite[Proposition 2.4]{LY}. As explained in \cite[Lemma 1.5]{L3}, we only need to deal with the case that $\gd(V)$ is finite. Since the conclusion holds for $V$ with $\gd(V) = 0$, we assume that $\gd(V) > 0$. From the surjection $\bfd P \to \bfd V \to 0$ one deduces that $\gd(\bfd V) \leqslant \gd(\bfd P) = \gd(P) -1$.

On the other hand, there is an object $\bfn$ such that $\deg(\bfn) = \gd(V) \geqslant 1$ and $(H_0(V))_{\bfn} \neq 0$. Now let $V'$ be the submodule of $V$ generated by $V_{\bft}$ with $\bft \leqslant \bfn$ and $\bft \neq \bfn$. Then $V'$ is a proper submodule of $V$ since otherwise one should have $(H_0(V))_{\bfn} = 0$. The surjection $V \to V/V' \to 0$ induces a surjection $\bfd V \to \bfd (V/V') \to 0$, and hence $\gd(\bfd V) \geqslant \gd(\bfd(V/V'))$. Note that the value of $\bfd(V/V')$ on an object $\bft$ with $\deg(\bft) = \deg(\bfn) - 1$ is nonzero, and the values of $\bfd(V/V')$ on all objects smaller than $\bft$ is 0. Therefore, $\gd(\bfd V) \geqslant \gd(\bfd(V/V')) \geqslant \deg{\bft} = \gd(V) - 1$.

(4): Immediately follows from Lemma \ref{self-embedding functors commute}.

(5): There are two cases:

When $i = j$, by the definition,
\begin{align*}
(\Sigma_i D_i V)_{\bfn} & = V_{\bfn + \mathbf{2}_i} / \pi_{\bfn + \bfo_i, i} \cdot V_{\bfn + \bfo_i};\\
(D_i \Sigma_i V)_{\bfn} & = V_{\bfn + \mathbf{2}_i} / \iota_i (\pi_{\bfn, i}) \cdot V_{\bfn + \bfo_i}
\end{align*}
where $\iota_i$ and $\pi_{\bfn, i}$ are defined in (\ref{iota}) and (\ref{pi}) respectively. By these definitions, both $\pi_{\bfn + \bfo_i, i}$ and $\iota_i (\pi_{\bfn, i})$ are maps:
\begin{equation*}
\bfn + \bfo_i = ([n_1], \, \ldots, \, [n_i + 1], \, \ldots, \, [n_m]) \to ([n_1], \, \ldots, \, [n_i + 2], \, \ldots, \, [n_m]) = \bfn + \mathbf{2}_i.
\end{equation*}
Moreover, for $j \neq i$, their $j$-th components are all identity maps. While their $i$-th components are
\begin{align*}
[n_i + 1] \to [n_i + 2], \quad & t \mapsto t+1;\\
[n_i + 1] \to [n_i + 2], \quad & t \mapsto
\begin{cases}
1, & t = 1;\\
t+1, & 2 \leqslant t \leqslant n_i+1.
\end{cases}
\end{align*}
respectively. Now the reader can see that the family of bijections
\begin{equation*}
(e_1, \, \ldots, \, e_{i-1}, \, g, \, e_{i+1}, \, \ldots, \, e_m) \in S_{\bfn + \mathbf{2}_i},
\end{equation*}
where $e_j$ is the identity element in $S_j$, and $g \in S_{n_i + 2}$ permutes $1$ and $2$ and fixes all other elements in $[n_i + 2]$, gives an isomorphism between $\Sigma_i D_i$ and $D_i \Sigma_i$.

When $i \neq j$, by the definition,
\begin{align*}
(\Sigma_i D_j V)_{\bfn} & = (D_j V)_{\bfn + \bfo_i} = V_{\bfn + \bfo_i + \bfo_j} / \pi_{\bfn + \bfo_i, j} \cdot V_{\bfn + \bfo_i};\\
(D_j \Sigma_i V)_{\bfn} & = V_{\bfn + \bfo_i + \bfo_j} / \iota_i (\pi_{\bfn, j}) \cdot V_{\bfn + \bfo_i}
\end{align*}
where $\iota_i$ and $\pi_{\bfn, i}$ are defined in (\ref{iota}) and (\ref{pi}) respectively. Now it is routine to check that $\pi_{\bfn + \bfo_i, i}$ and $\iota_i (\pi_{\bfn, i})$ are the same map from $\bfn + \bfo_i$ to $\bfn + \bfo_i + \bfo_j$. Explicitly, for $s \neq j$, the $s$-th component of this map is the identity, while its $j$-th component sends elements in $[n_j]$ to $[n_j+1]$ by adding 1. Consequently, in this case $D_j \Sigma_i = \Sigma_i D_j$.
\end{proof}

Although $\bfs$ is a direct sum of $\Sigma_i$ for $i \in [m]$, in general we do not have a good relationship between homological degrees of $V$ and those of $\Sigma_i V$ and $D_i V$, except that $\gd(D_iV) < \gd(V)$ and $\gd(D_i V) \leqslant \gd(\Sigma_iV) \leqslant \gd(V)$ always hold. Here is an example.

\begin{example} \normalfont
Let $m = 2$, $s$ be a positive integer, and $V$ be a $\C$-module satisfying the following conditions: $\gd(V) = s$, and $V_{\bfn} = 0$ whenever $n_2 > 0$. Then $\Sigma_2 V = D_2 V = 0$, so $\gd(V) - \gd(\Sigma_2 V) = s + 1$ can be as large as we want.
\end{example}

However, under a certain weak condition, we can still obtain a good control on the difference $\gd(V) - \gd(D_i V)$ for $i \in [m]$.

\begin{lemma} \label{a special case}
Let $V$ be a nonzero $\C$-module such that $V_{\bfn} = 0$ for all objects $\bfn$ satisfying $n_i = 0$. Then
\begin{equation*}
\gd(D_iV) \leqslant \gd(\Sigma_i V) \leqslant \gd(V) = \gd(D_i V) + 1.
\end{equation*}
\end{lemma}

\begin{proof}
From the given condition we know that $\gd(V) \geqslant 1$. Now one can copy the proof of Statement (3) of Lemma \ref{basic properties of sigma}, noting that in that argument the object $\bfn$ satisfies $n_i \geqslant 1$, so one can choose $\bft = \bfn - \bfo_i$.
\end{proof}

\subsection{Torsion modules}

Recall that an element $v \in V_{\bfn}$ for a certain object $\bfn$ is a torsion element if there exist another object $\bft$ and an injection $\bff \in \C (\bfn, \bft)$ such that $\bff \cdot v = 0$. Let $V_T$ be the submodule of $V$ generated by all these torsion elements, and denote $V_F = V/V_T$. We say that $V$ is \emph{torsion} if $V = V_T$. If $V_T = 0$, $V$ is called \emph{torsion free}. In particular, all free modules $M (\bfn)$ are torsion free by the second statement of Lemma \ref{basic combinatorics}.

\begin{lemma} \label{torsion modules}
Let $V$ be a $\C$-module. Then:
\begin{enumerate}
\item $V$ is a torsion module if and only if for very object $\bfn$ and every $v \in V_{\bfn}$, $v$ is a torsion element. In particular, submodules and quotient modules of torsion modules are still torsion.
\item $K_iV = \bigoplus _{\bfn \in \Nm} \{v \in V_{\bfn} \mid \pi_{\bfn, i} \cdot v = 0 \} = \bigoplus _{\bfn \in \Nm} \{ v \in V_{\bfn} \mid \bff \cdot v  = 0, \, \forall \bff \in \C (\bfn, \bfn + \bfo_i) \}$.
\item $V$ is torsion free if and only if $\bfk V = 0$.
\item If $V$ is torsion free, so is $\bfs V$.
\item In a short exact sequence $0 \to U \to V \to W \to 0$, if both $U$ and $W$ are torsion free, so is $V$.
\item In a short exact sequence $0 \to U \to V \to W \to 0$, if $W$ is torsion free, then the sequence $0 \to \bfd U \to \bfd V \to \bfd W \to 0$ is exact as well.
\item If $V$ is torsion, so is $\Sigma_i V$ for $i \in [m]$.
\end{enumerate}
\end{lemma}

\begin{proof}
(1): One direction of the first part is trivial. For the other direction, suppose that $V$ is a torsion module, so it is generated by torsion elements. Therefore, it suffices to prove the following claim: if $v \in V_{\bfn}$ is a torsion element and $\bff: \bfn \to \bft$ is an injection, then $\bff \cdot v \in V_{\bft}$ is a torsion element as well. Since $v$ is a torsion element, there exist an object $\bfl > \bfn$ and an injection $\bfg \in \C (\bfn, \bfl)$ such that $\bfg \cdot v = 0$. By Lemma \ref{basic combinatorics}, every injection in $\C (\bfn, \bfl)$ sends $v$ to 0, so does every morphism in $\C (\bfn, \bfl + \bft)$. Taking an $\bfh \in \C (\bft, \bfl + \bft)$, we know $(\bfh \circ \bff) \cdot v = 0$. That is, $\bff \cdot v$ is a torsion element.

The second part immediately follows from the first one.

(2): By the definition, the natural map $V \to \Sigma_i V$, while restricted to a fixed object $\bfn$, is given by $v \mapsto \pi_{\bfn, i} \cdot v$, and hence the first equality holds. For the second equality, one just notices that $\pi_{\bfn, i} \in \C (\bfn, \bfn + \bfo_i)$ and this morphism set has only one orbit as an $S_{\bfn + \bfo_i}$-set.

(3): If $\bfk V \neq 0$, then there exists a certain $i \in [m]$ such that $K_i V \neq 0$. By the above description, $V$ contains nonzero torsion elements, and hence is not torsion free. Conversely, if $V$ is not torsion free, we can find two objects $\bfn < \bft$, a nonzero element $v \in V_{\bfn}$, and a morphism $\bff \in \C(\bfn, \bft)$ such that $\bff \cdot v  = 0$. By a simple induction one can assume that $\deg(\bft) - \deg(\bfn) = 1$; that is, there is a certain $i \in [m]$ such that $\bft  = \bfn + \bfo_i$. Clearly, $v \in K_i V$, so $\bfk V \neq 0$.

(4): Since $V$ is torsion free, there is a short exact sequence
\begin{equation*}
\xymatrix{
0 \ar[r] & V \ar[r]^-{\theta_V} & \bfs V \ar[r] & \bfd V \ar[r] & 0.
}
\end{equation*}
Applying the exact functor $\bfs$ one gets a short exact sequence
\begin{equation*}
\xymatrix{
0 \ar[r] & \bfs V \ar[r]^-{\bfs \theta_V} & \bfs \bfs V \ar[r] & \bfs \bfd V \ar[r] & 0.
}
\end{equation*}
Note that in the above sequence, the map $\bfs \theta_V$ might not be the natural map $\theta_{\bfs V}: \bfs V \to \bfs (\bfs V)$. However, since we just proved that $\bfd \bfs \cong \bfs \bfd$ in Lemma \ref{basic properties of sigma}, this sequence is actually isomorphic to the exact sequence determined by the map $\theta_{\bfs V}$. In particular, $\bfk \bfs V \cong \bfs \bfk V = 0$. That is, $\bfs V$ is torsion free as well. One can also directly check that $\bfs V$ is torsion free by applying the argument in the proof of \cite[Proposition 2.1]{LY}.

(5): Applying the exact functor $\bfs$ one gets a commutative diagram where all vertical rows represent natural maps:
\begin{equation*}
\xymatrix{
0 \ar[r] & U \ar[r] \ar[d] & V \ar[r] \ar[d] & W \ar[r] \ar[d] & 0 \\
0 \ar[r] & \bfs U \ar[r] & \bfs V \ar[r] & \bfs W \ar[r] & 0.
}
\end{equation*}
According to the third statement, the maps $U \to \bfs U$ and $W \to \bfs W$ are injective, so is the map $V \to \bfs V$ by the snake Lemma.

(6): Follows from the above commutative diagram and the snake Lemma.

(7): Suppose that $\Sigma_i V \neq 0$ and take an element $v \in (\Sigma_i V)_{\bfn} = V_{\bfn + \bfo_i}$. Since $V$ is torsion, there exist an object $\bft > \bfn + \bfo_i$ and a morphism $\bff \in \C (\bfn + \bfo_i, \bft)$ such that $\bff \cdot v = 0$. Consequently, every morphism in $\C (\bfn + \bfo_i, \bft)$ sends $v$ to 0. Note that the self-embedding functor $\iota_i$ maps morphisms in $\C (\bfn, \bft - \bfo_i)$ into $\C(\bfn + \bfo_i, \bft)$ injectively. Therefore, every morphism in $\C (\bfn, \bft - \bfo_i)$ sends $v$ to 0, where $v$ is regarded as an element in $\Sigma_i V$. That is, $v$ is still a torsion element in $\Sigma_i V$.
\end{proof}

\begin{remark} \normalfont \label{description of KV}
Intuitively, an element $v \in V_{\bfn}$ is contained in $K_iV$ if and only if it vanishes when moving one step along the $i$-th coordinate axis direction. Therefore, $K_iV$ is a direct sum of several direct summands, each of which is supported on a ``hyperplane" (when $m = 1$, a hyperplane is just a point) perpendicular to the $i$-th coordinate axis.

One may give a more conceptual description for $K_iV$. Let $I_i$ be the free $k$-module spanned by all morphisms in $\C (\bfn, \bft)$, $\bfn, \bft \in \Nm$, such that $t_i > n_i$. The reader can check that $I_i$ is a two-sided ideal of $k \C$. Moreover, $K_i V$ is precisely the \emph{trace} of $k \C / I_i$ in $V$. That is,
\begin{equation*}
K_i V = \sum _{\begin{matrix} \alpha \in \Hom _{k \C} (M(\bfn)/I_iM(\bfn), V) \\ \bfn \in \Nm \end{matrix}} \alpha (V).
\end{equation*}
\end{remark}

Now we define \emph{torsion degree}, an invariant playing a crucial role in this paper.

\begin{definition} \normalfont \label{torsion degree}
For $i \in [m]$, let
\begin{equation*}
\td_i(V) = \sup \{ s \mid \exists \bfn \text{ such that } (K_iV)_{\bfn} \neq 0 \text{ and } n_i = s \}.
\end{equation*}
The torsion degree of $V$ is set to be
\begin{equation*}
\td(V) = \sup \{ \td_i(V) \mid i \in [m] \}.
\end{equation*}
Whenever a set on the right side is empty, we set the number on the left side to be -1.
\end{definition}

\begin{remark} \normalfont
For $m = 1$, this is not the usual definition of torsion degrees in \cite{L2, LR, LY}. However, as shown in \cite{L1}, the above one is an equivalent definition. Explicitly, for $i \in [m]$, $\td_i(V) < \infty$ if and only if the following two conditions hold: there is an object $\bfn$ with $n_i = \td_i(V)$ such that we can find an element $0 \neq v \in V_{\bfn}$ satisfying $\alpha \cdot v = 0$ for all $\alpha \in \C (\bfn, \bfn + \bfo_i)$; for any object $\bft$ with $\bft_i > \td_i(V)$ and any nonzero $v \in V_{\bft}$ and $\alpha \in \C(\bft, \bft + \bfo_i)$, $\alpha \cdot v \neq 0$.

It is also clear that $\td(V) = \td(V_T)$ since $V$ and $V_T$ have the same torsion elements, which completely determine the torsion degree.
\end{remark}

At this moment it is not clear that $\td(V)$ is a finite number. However, we have:

\begin{lemma} \label{finite torsion degree}
Let $V$ be a finitely generated nonzero $\C$-module. One has:
\begin{enumerate}
\item If $V'$ is a submodule, then
\begin{equation*}
\td(V) \leqslant \max \{\td(V'), \, \td(V/V') \}.
\end{equation*}
\item If $V$ is Noetherian, then $\td(V)$ is finite and $\td_i(\Sigma_i V) \leqslant \td_i(V) - 1$ whenever $\td_i(V) \neq 0$. In particular, $\td(\Sigma_1 \ldots \Sigma_m V) \leqslant \td(V) - 1$ whenever $\td(V) \neq 0$.
\end{enumerate}
\end{lemma}

\begin{proof}
(1): Firstly suppose that $\td(V)$ is finite, there exists a certain $i \in [m]$ such that $\td(V) = \td_i(V)$. Now by the previous remark, there is an object $\bfn$ with $n_i = \td_i(V)$ such that we can find an element $0 \neq v \in V_{\bfn}$ satisfying $\alpha \cdot v = 0$ for all $\alpha \in \C (\bfn, \bfn + \bfo_i)$. If $v \in V'_{\bfn}$, then $\td_i(V') \geqslant \td_i(V)$, so $\td(V') \geqslant \td(V)$. Otherwise, the image $0 \neq \bar{v} \in (V/V')_{\bfn}$ of $v$ satisfies $\alpha \cdot v = 0$ for all $\alpha \in \C (\bfn, \bfn + \bfo_i)$, so by the same argument, $\td_i(V/V') \geqslant \td_i(V)$, and hence $\td(V/V') \geqslant \td(V)$.

If $\td(V) = \infty$, there exists a certain $i \in [m]$ such that $\td_i(V) = \infty$. In particular, we can find an infinite sequence of objects $\bfn^j$ and an infinite sequence of nonzero elements $v^j \in V_{\bfn^j}$ such that $\alpha \cdot v^j = 0$ for all $\alpha \in \C (\bfn^j, \bfn^j + \bfo_i)$ and $\bfn^1_i < \bfn^2_i < \ldots$. Using the same argument, one can show that either $\td_i(V) = \infty$ or $\td_i(V/V') = \infty$.

(2): To show the finiteness of $\td(V)$, by Definition \ref{torsion degree}, it suffices to show that $\td_i(V) < \infty$ for $i \in [m]$. By Statement (2) of Lemma \ref{torsion modules} and Remark \ref{description of KV}, $K_i V$ is a direct sum of a few direct summands, each of which is supported on a ``hyperplane" in $\Nm$ which is perpendicular to the $i$-th coordinate axis. Since $V$ is Noetherian, $K_iV$ is finitely generated as well. Therefore, the number of such direct summands is finite. However, $\td_i(V)$ is nothing but the maximal height of these ``hyperplanes", so must be finite as well.

It is clear that $\td_i(\Sigma_i V) = \td_i(V) - 1$ whenever $K_i V \neq 0$. Moreover, one observes that $\td_i (\Sigma_j V) \leqslant \td_i(V)$. These observations can be easily verified by Remark \ref{description of KV}.
\end{proof}

\begin{remark} \normalfont
When $m = 1$, one has $\gd(\bfk V) = \td(V)$; see \cite{L2}. For $m > 1$, the equality might not hold, and the reader can easily find an example such that $\gd(KV) > \td(V)$. However, we always have $\gd(K_i V) \geqslant \td_i(V)$. To see this, consider the direct summand of $K_iV$ which is supported on the highest hyperplane. Then $\td_i(V)$ is precisely the height of this hyperplane, while the generating degree of this direct summand is greater than or equal to this height since the degree of every object in this hyperplane is at least the height. Therefore, as claimed, $\gd(K_iV) \geqslant \td_i(V)$ for each $i \in [m]$. Consequently, $\gd(\bfk V) \geqslant \td(V)$.

For $m = 1$, one has $\td(\bfs V) = \td(V) - 1$ whenever $V$ is not torsion free. This equality does not hold for $m > 1$ as well. For example, let $m = 2$, and $V$ be a nonzero $\C$-module such that its values on objects different from $(3, 3)$ are all zeroes. Then $\td(\bfs V) = \td(V) = 3$. One cannot strengthen the inequality in the above lemma to an equality, either. Indeed, let $m = 2$ and $V$ be a nonzero module such that its values on all objects different from $(3, 0)$ are zeroes. Then $\Sigma_1 \Sigma_2 V = 0$ and hence $\td(V) - \td(\Sigma_1 \Sigma_2 V) = 4$.
\end{remark}

\section{Noetherianity}

Throughout this section let $k$ be a commutative Noetherian ring. We will show the following statement: every finitely generated $\C$-module is also finitely presented. By a standard homological argument, this statement is equivalent to the locally Noetherian property of $\C$.

\subsection{Truncation functors.}

We introduce the following definition.

\begin{definition} \label{locally finite}
A $\C$-module $V$ is \emph{locally finite} if for each object $\bfn$, $V_{\bfn}$ is a finitely generated $k$-module.
\end{definition}

Since $k$ is Noetherian, the reader can see that the category of all locally finite $\C$-modules is an abelian category. Furthermore, a locally finite $\C$-module is finitely generated (resp., finitely presented) if and only if its generating degree is finite (resp., generating degree and first homological degree are finite).

From now on we only consider locally finite $\C$-modules. We have:

\begin{lemma} \label{equivalent characterizations}
The following statements are equivalent:
\begin{enumerate}
\item The category $\C$ is locally Noetherian.
\item Every finitely generated $\C$-module is finitely presented.
\item The first homological degree of every finitely generated module is finite.
\end{enumerate}
\end{lemma}

\begin{proof}
Obvious.
\end{proof}

For each $i \in [m]$, we define
\begin{equation*}
\mJ_i = k \bigsqcup_{n_i > 0} \C(\bft, \bfn),
\end{equation*}
spanned by morphism ending at objects $\bfn$ with $n_i > 0$. This is a two-sided ideal of the category algebra $k \C$. Furthermore, the reader can check that the quotient algebra $k \C / \mJ_i$ is isomorphic to the category algebra $k \C'$, where $\C'$ is the skeletal category of $\FI^{m-1}$ whose objects are parameterized by elements in $\mathbb{N}^{m-1}$. By convention, if $m = 1$, then $k \C /\mJ_i \cong k$.

Let $\tau_i = k \C / \mJ_i \otimes_{k \C} -$, which is a functor from the category of $\C$-modules to the category of $\C'$-modules (by identifying $k \C / \mJ_i$ with $k \C'$). Conversely, every $\C'$-module can be viewed as a $\C$-module by \emph{lifting}. The behavior of $\tau_i$ has a very explicit description. That is, given a $\C$-module $V$, one has
\begin{equation*}
(\tau_i V)_{\bfn} =
\begin{cases}
V_{\bfn}, & \text{ if } n_i = 0;\\
0, & \text{ otherwise},
\end{cases}
\end{equation*}
and
\begin{equation*}
(\mJ_i V)_{\bfn} =
\begin{cases}
V_{\bfn}, & \text{ if } n_i > 0;\\
0, & \text{ otherwise}.
\end{cases}
\end{equation*}
Therefore, we obtain a short exact sequence $0 \to \mJ_i V \to V \to \tau_i V \to 0$. Moreover, one has $\Sigma_i V = \Sigma_i \mJ_i V$ since $\Sigma_i \tau_i V = 0$.

\subsection{Finitely presented property.}

Let $V$ be a finitely generated $\C$-module. In this subsection we relate the finitely presented property of $V$ to that of $\Sigma_iV$ for each $i \in [m]$.

\begin{lemma} \label{finitely presented property}
Suppose that $\FI^{m-1}$ is locally Noetherian over $k$, and let $V$ be a finitely generated $\C$-module.
\begin{enumerate}
\item If there is a certain $i \in [m]$ such that $\Sigma_i V$ is finitely presented, then $V$ is finitely presented.
\item If the natural map $V \to \Sigma_i V$ is injective, and $D_iV$ is finitely presented, then $V$ is finitely presented.
\end{enumerate}
\end{lemma}

\begin{proof}
By the exact sequence $0 \to \mJ_i V \to V \to \tau_i V \to 0$, it suffices to show that both $\tau_iV$ and $\mJ_i V$ are finitely presented. Note that $\tau_i V$ is finitely generated, and so is $\mJ_i V$ because
\begin{equation*}
\gd(\mJ_i V) \leqslant \gd(\Sigma_i \mJ_i V) + 1 = \gd(\Sigma_i V) + 1 \leqslant \gd(V) + 1
\end{equation*}
by Lemma \ref{a special case} (the first inequality) and Lemma \ref{basic properties of sigma} (the second inequality).

We claim that $\hd_1(\tau_i V)$ is finite. Let $P \to \tau_i V \to 0$ be a surjection such that $P$ is a finitely generated projective $\C$-module with $\gd(P) = \gd(\tau_i V)$. It induces the following commutative diagram
\begin{equation*}
\xymatrix{
 & 0 \ar[r] & \mJ_i P \ar[r] \ar[d] & \tilde{W} \ar[r] \ar[d] & W \ar[r] & 0\\
 & & P \ar@{=}[r] \ar[d] & P \ar[d] \\
0 \ar[r] & W \ar[r] & \tau_i P \ar[r] & \tau_i V \ar[r] & 0.
}
\end{equation*}
Since $\gd(P)$ is finite, so is $\gd(\mJ_i P)$ by the same argument as in the previous paragraph. Moreover, by viewing terms in the bottom row as $\FI^{m-1}$-modules and using the given condition, one deduces that $\gd(W) < \infty$. Consequently, from the top row we conclude that $\gd(\tilde{W})$ is finite, so $\tau_i V$ is finitely presented.

Now we turn to $\mJ_i V$. Let $0 \to U \to Q \to \mJ_i V \to 0$ be a short exact sequence such that $Q$ is a finitely generated projective $\C$-module with $\gd(Q) = \gd(\mJ_i V)$.

(1): Applying $\Sigma_i$ we get $0 \to \Sigma_i U \to \Sigma_i Q \to \Sigma_i \mJ_i V = \Sigma_i V \to 0$. Since $\Sigma_i V$ is supposed to be finitely presented, $\gd(\Sigma_i U) < \infty$. Moreover, by the structure of $\mJ_i V$, we can suppose that $Q_{\bfn} = 0$ for any object $\bfn$ with $n_i = 0$. Then by Lemma \ref{a special case}, we conclude that $\gd(U) \leqslant \gd(\Sigma_i U) + 1$. That is, $\mJ_i V$ is finitely presented as well.

(2): Since the natural map $V \to \Sigma_i V$ is injective, so is the natural map $\mJ_i V \to \Sigma_i \mJ_i V = \Sigma_i V$. Therefore, we have the following commutative diagram
\begin{equation*}
\xymatrix{
 & 0 \ar[r] & \mJ_i V \ar[r] \ar[d] & V \ar[r] \ar[d] & \tau_i V \ar[r] & 0\\
 & & \Sigma_i \mJ_i V \ar@{=}[r] \ar[d] & \Sigma_i V \ar[d] \\
0 \ar[r] & \tau_i V \ar[r] & D_i \mJ_i V \ar[r] & D_iV \ar[r] & 0.
}
\end{equation*}
Note that $D_iV$ is finitely presented by the given condition, and we just proved that $\tau_i V$ is finitely presented. Consequently, $D_i \mJ_i V$ is finitely presented, too.

By Statement (6) of Lemma \ref{torsion modules}, the short exact sequence $0 \to U \to Q \to \mJ_i V \to 0$ induces a short exact sequence $0 \to D_i U \to D_i Q \to D_i \mJ_i V \to 0$. Since $D_i \mJ_i V$ is finitely presented, $\gd(D_i U) < \infty$. By Lemma \ref{a special case}, $\gd(U) < \infty$. That is, $\mJ_i V$ is finitely presented.
\end{proof}

\subsection{Filtrations}

In this subsection we construct a sequence of quotient modules for a $\C$-module $V$ using a fixed shift functor $\Sigma_1$. Let $V^0 = V$, $V^1$ be the image of the natural map $V^0 \to \Sigma_1 V^0$, and $V^2$ be the image of the natural map $V^1 \to \Sigma V^1$, and so on. In this way we obtain a sequence of quotient maps
\begin{equation*}
V = V^0 \to V^1 \to V^2 \to \ldots.
\end{equation*}
Define
\begin{equation*}
\bar{V} = \lim _{\to} V^i,
\end{equation*}
which is a quotient module of $V$.

\begin{lemma} \label{filtration}
Let $\bar{V}$ be as define above. Then the natural map $\bar{V} \to \Sigma_1 \bar{V}$ is injective. Moreover, if the kernel of $V \to \bar{V}$ is finitely generated, then the above sequence stabilizes after finitely many steps.
\end{lemma}

\begin{proof}
The proof is similar to proofs of \cite[Lemmas 3.1, 3.3 and 3.7]{GL2}.
\end{proof}

\subsection{Proof of Noetherianity}

Now we are ready to prove the locally Noetherianity of $\C$ over commutative Noetherian rings.

\begin{proof}[A proof of Theorem \ref{noetherianity}]
We use a double induction on $m$ and the generating degree $\gd(V)$. For $m = 0$, $k \FI^0 = k$ by our convention, so the conclusion holds. Suppose that the conclusion holds for $\FI^s$ with $s < m$. For $\FI^m$, it suffices to show that for any finitely generated $\C$-module $V$ we have $\hd_1(V) < \infty$ by Lemma \ref{equivalent characterizations}. Clearly, this is true for $\gd(V) = -1$, that is, $V = 0$, so we suppose that $\gd(V) \geqslant 0$.

As described in the previous subsection, we obtain a sequence of quotient maps
\begin{equation*}
V = V^0 \to V^1 \to V^2 \to \ldots
\end{equation*}
and a limit $\bar{V}$. If $\gd(\bar{V}) < \gd(V)$, then by the induction hypothesis on generating degrees, $\bar{V}$ is finitely presented. If $\gd(\bar{V}) = \gd(V)$, then the natural map $\bar{V} \to \Sigma_1 \bar{V}$ is injective, and $\gd(D_1 \bar{V}) < \gd(\bar{V}) \leqslant \gd(V)$. Therefore, by the induction hypothesis on generating degrees, $D_1 \bar{V}$ is finitely presented, so does $\bar{V}$ by Statement (2) of Lemma \ref{finitely presented property}. In both cases, we know that $\bar{V}$ is finitely presented, so the kernel of the quotient map $V \to \bar{V}$ is finitely generated. By Lemma \ref{filtration}, there exists an $N \in \mathbb{N}$ such that $V^n = V^N$ for all $n \geqslant N$.

By our construction, there is a short exact sequence $0 \to V^N \to \Sigma_1 V^{N-1} \to D_1 V^{N-1} \to 0$. The induction hypothesis tells us that $D_1 (V^{N-1})$ is finitely presented since $\gd(D_1 V^{N-1}) < \gd(V^{N-1}) \leqslant \gd(V)$. We just proved that $V^N$ is finitely presented, so is $\Sigma_1 V^{N-1}$. By Statement (1) of Lemma \ref{finitely presented property}, $V^{N-1}$ is finitely presented as well. Replacing $V^N$ by $V^{N-1}$ and carrying out this procedure, recursively one can show that all $V^i$, $0 \leqslant i \leqslant N$ are finitely presented. In particular, $V$ is finitely presented. The conclusion follows by induction.
\end{proof}

\section{Relative projective modules}

In this section we consider relative projective modules. These special modules generalize projective modules, and have played a key role in constructing a homological computation machinery for $\FI$-modules. For details, see \cite{L2, LR, LY}. Our goal is to extend most results in \cite{LY} from $\FI$ to $\FI^m$ for arbitrary $m \geqslant 1$. Let $k$ be an arbitrary commutative ring, and $V$ be a $\C$-module.

\subsection{Basic properties} Given an object $\bfn$ and a $k S_{\bfn}$-module $W$, we can define a $\C$-module $M(\bfn) \otimes _{k S_{\bfn}} W$ as $M(\bfn)$ is a $(k \C, k S_{\bfn})$-bimodule. We call it a \emph{basic} relative projective module.

\begin{definition} \normalfont \label{relative projective modules}
A $\C$-module $V$ is a relative projective module if it has a filtration $0 = V^0 \subseteq V^1 \subseteq \ldots \subseteq V^n = V$ such that each factor $V^{i+1} / V^i$, $0 \leqslant i \leqslant n-1$, is isomorphic to a basic relative projective module.
\end{definition}

The reader immediately sees that if $V$ is a relative projective module, then $\gd(V) < \infty$. Of course, one can remove this restriction from the above definition. However, since we are mostly interested in finitely generated $\C$-modules, it is reasonable to impose this condition. It is also obvious that projective modules with finite generating degrees are relative projective modules.

\begin{lemma} \label{properties of basic relative projective modules}
Let $V$ be a $\C$-module generated by its value on a certain object $\bfn$. One has:
\begin{enumerate}
\item The following are equivalent:
\begin{itemize}
\item $V$ is a basic relative projective module;
\item $H_s(V) = 0$ for all $s \geqslant 1$;
\item $H_1(V) = 0$.
\end{itemize}
\item If $V$ is relative projective, then it is torsion free.
\item If $\hd_1(V) \leqslant \gd(V)$, then $V$ is relative projective.
\item If $V$ is relative projective, so is $\Sigma_i V$ and $D_i V$ for $i \in [m]$.
\end{enumerate}
\end{lemma}

\begin{proof}
(1): Suppose that $V$ is basic relative projective module. That is, $V \cong M(\bfn) \otimes _{k S_{\bfn}} V_{\bfn}$. Take a short exact sequence $0 \to W \to P \to V_{\bfn} \to 0$ of $k S_{\bfn}$-modules such that $P$ is projective. Applying the functor $M(\bfn) \otimes _{k S_{\bfn}} -$ which is exact by Statement (1) of Lemma \ref{basic combinatorics}, we get a short exact sequence
\begin{equation*}
0 \to M(\bfn) \otimes _{k S_{\bfn}} W \to M(\bfn) \otimes _{k S_{\bfn}} P \to M(\bfn) \otimes _{k S_{\bfn}} V_{\bfn} \cong V \to 0.
\end{equation*}
Applying the functor $k \C / \mm \otimes_{k \C} -$ we recover the original short exact sequence. That is, $H_1(V) = 0$. Replacing $V$ by $V' = M(\bfn) \otimes _{k S_{\bfn}} W$ we deduce that $H_2(V) = 0$. Recursively, for every $s \geqslant 1$, one gets $H_s(V) = 0$.

Conversely, suppose that $H_1(V) = 0$. Since $V$ is generated by $V_{\bfn}$, there is a short exact sequence
\begin{equation*}
0 \to K \to M(\bfn) \otimes _{k S_{\bfn}} V_{\bfn} \to V \to 0.
\end{equation*}
The long exact sequence
\begin{equation*}
\ldots \to H_1(V) = 0 \to H_0(K) \to H_0 (M(\bfn) \otimes _{k S_{\bfn}} V_{\bfn}) = V_{\bfn} \to H_0(V) = V_{\bfn} \to 0
\end{equation*}
implies that $H_0(K) = 0$. That is, $K = 0$, and hence $V$ is relative projective.

(2): Since $V$ is relative projective, without loss of generality we assume that $V = M(\bfn) \otimes _{k S_{\bfn}} V_{\bfn}$. Take an arbitrary $\bfl \geqslant \bfn$, an element $0 \neq v \in V_{\bfl}$, and a morphism $\bff \in \C (\bfl, \bft)$ with $\bft > \bfl$. We want to show that $\bff \cdot v \neq 0$.

Note that $V_{\bfl} = \C(\bfn, \bfl) \otimes_{k S_{\bfn}} V_{\bfn}$, and $\C(\bfn, \bfl)$ is a right free $k S_{\bfn}$-module by Statement (1) of Lemma \ref{basic combinatorics}. Therefore, $v$ can be written as $(\bfg_1 \otimes u_1) + \ldots + (\bfg_s \otimes u_s)$ such that each $u_i \in V_{\bfn}$ is nonzero, and moreover $\bfg_i$ and $\bfg_j$ are contained in distinct orbits of the right $S_{\bfn}$-set $\C (\bfn, \bfl)$ if $i \neq j$. Consequently,
\begin{equation*}
\bff \cdot v = (\bff \bfg_1 \otimes u_1) + \ldots + (\bff \bfg_s \otimes u_s).
\end{equation*}

We claim that $\bff \bfg_i$ and $\bff \bfg_j$ are contained in distinct orbits of the right $S_{\bfn}$-set $\C (\bfn, \bft)$ if $i \neq j$. Indeed, if $\bff \bfg_i$ and $\bff \bfg_j$ are in the same orbit, there exists an automorphism $\bfh \in S_{\bfn}$ such that $\bff \bfg_i \bfh = \bff \bfg_j$. By Statement (2) of Lemma \ref{basic combinatorics}, $\bfg_i \bfh = \bfg_j$, contradicting the assumption imposed on $\bfg_i$ and $\bfg_j$. Therefore, those $\bff \bfg_i \otimes u_i$ are nonzero and lie in distinct direct summands of $\C (\bfn, \bfl) \otimes_{k S_{\bfn}} V_{\bfn}$. Consequently, $\bff \cdot v \neq 0$.

(3): Again, consider the short exact sequence
\begin{equation*}
0 \to K \to M(\bfn) \otimes _{k S_{\bfn}} V_{\bfn} \to V \to 0
\end{equation*}
and its induced long exact sequence
\begin{equation*}
\ldots \to 0 \to H_1(V) \to H_0(K) \to H_0 (M(\bfn) \otimes _{k S_{\bfn}} V_{\bfn}) = V_{\bfn} \to H_0(V) = V_{\bfn} \to 0.
\end{equation*}
Since $\hd_1(V) \leqslant \gd(V)$, we know that $H_0 (K)$ is only supported on objects whose degrees are at most $\gd(V) = \deg {\bfn}$. But from the short exact sequence we see that $K$ is only supported on objects strictly greater than $\bfn$. The only possibility is that $K = 0$.

(4): Consider a short exact sequence $0 \to K \to P \to V \to 0$ such that $P$ is a projective $k \C$-module generated by $P_{\bfn}$. By the previous arguments we know that all terms in it are basic relative projective modules generated by their values on $\bfn$, and hence are torsion free. By Statement (6) of Lemma \ref{torsion modules}, for each $i \in [m]$, we get a short exact sequence $0 \to D_iK \to D_iP \to D_i V \to 0$. Since $D_i V = 0$ whenever $n_i = 0$, without loss of generality we assume that $n_i > 0$. Then $D_i V$ is generated by its value on the object $\bfn - \bfo_i$ since so is $D_iP$. Replacing $V$ by $K$ one knows that $D_i K$ is also generated by its value on $\bfn - \bfo_i$. Since $D_i P$ is projective, the long exact sequence of homology groups induced by $0 \to D_i K \to D_i P \to D_i V \to 0$ tells us that $\gd(D_i V) = \gd(D_i P) \geqslant \gd (D_iK) \geqslant \hd_1(D_iV)$. By the previous statement, $D_iV$ is relative projective. But $D_i V = \Sigma_i V / V$, so $\Sigma_i V$ is relative projective as well since so is $V$.
\end{proof}

\begin{remark} \normalfont
For $m = 1$, Statements (2) and (4) have been verified in \cite[Lemma 2.2]{N}. Actually, slightly modifying the argument there, one can prove the following stronger conclusion: If $V$ is a basic relative projective module generated by its value on a certain object $\bfn$, then for $i \in [m]$,
\begin{equation*}
\Sigma_i V \cong V \oplus (M(\bfn - \bfo_i) \otimes _{k S_{\bfn - \bfo_i}} V_{\bfn}),
\end{equation*}
whenever $n_i > 0$. For $n_i = 0$, $\Sigma_i V \cong V$.

From the proof we also observe the following fact: Let $V$ be a basic relative projective $\C$-module generated by $V_n$, and $W_n$ be a $k S_{\bfn}$-module. Then $H_0 (M(\bfn) \otimes_{k S_{\bfn}} W_n) \cong W_n$, and $M(\bfn) \otimes _{k S_{\bfn}} H_0(V) \cong V$.
\end{remark}

Since every relative projective module has a filtration whose factors are basic relative projective modules, from the above lemma one gets:

\begin{proposition} \label{properties of relative projective modules}
Let $V$ be a $\C$-module with $\gd(V) < \infty$. Then $V$ is relative projective if and only if $H_1(V) = 0$, and if and only if $H_s(V) = 0$ for all $s \geqslant 1$. Moreover, if $V$ is relative projective, then it is torsion free, and $\Sigma_i V$ and $D_i V$ are relative projective as well for $i \in [m]$.
\end{proposition}

\begin{proof}
The first statement is clear, and the first half of the second statement follows from Statement (5) of Lemma \ref{torsion modules}. To prove the second half of the second statement, suppose that $V^{\bullet}: 0 = V^0 \subseteq V^1 \subseteq \ldots \subseteq V^n = V$ is a filtration of $V$ such that each factor is a basic relative projective module. Applying the $i$-th shift functor $\Sigma_i$ we get a filtration $\Sigma_i V^{\bullet}$ of $\Sigma_i V$ such that each factor is still a relative projective module (might not be a basic relative projective module) by Statement (4) of the previous lemma. Moreover, since relative projective modules are torsion free, by Statement (6) of Lemma \ref{torsion modules}, applying $D_i$ we get a filtration $D_i V^{\bullet}$ of $D_i V$ such that each factor is a basic relative projective module.
\end{proof}

The following recursion lemma generalizes results in \cite[Lemmas 3.12 and 4.2]{LY}, and plays a vital role in the proofs of Theorems \ref{homological characterizations of relative projective modules} and \ref{relative projective complexes}.

\begin{lemma} \label{recursion}
Let $V$ be a torsion free $\C$-module with $\gd(V) < \infty$. One has:
\begin{enumerate}
\item If $\bfd V$ is relative projective, so is $V$.
\item If $V$ is relative projective, and $V'$ is relative projective submodule, then $V/V'$ is relative projective.
\end{enumerate}
\end{lemma}

\begin{proof}
The conclusions hold for $V = 0$ trivially, so we may assume that $V \neq 0$.

(1): Since $\gd(V)$ is finite, $H_0(V)$ is only supported on a set $S$ of finitely many objects in $\Nm$. We use induction on the cardinality $|S|$. The conclusion holds clearly if $|S| = 0$. For $|S| \geqslant 1$, we choose an object $\bfn \in S$ such that $\deg(\bfn)$ is maximal; that is, $\deg(\bfn) = \gd(V)$ (of course, this choice might not be unique). Let $V'$ be the submodule of $V$ generated by its values on objects in $S \setminus \{ \bfn \}$. Then $V'' = V / V'$ is nonzero and is generated by its value on $\bfn$.

Consider the short exact sequence $0 \to V' \to V \to V'' \to 0$. We claim that $V''$ is a basic relative projective module. To see this, from the long exact sequence of homology groups one has
\begin{equation*}
\hd_1(V'') \leqslant \max \{ \gd(V'), \, \hd_1(V) \} \leqslant \max \{ \gd(V), \, \hd_1(V) \}.
\end{equation*}
since $\gd(V') \leqslant \gd(V)$ by our construction of $V'$. Furthermore, let $0 \to W \to P \to V \to 0$ be a short exact sequence of $\C$-modules such that $P$ is a free $\C$-module satisfying $\gd(V) = \gd(P)$. Applying $\bfd$ we get another short exact sequence $0 \to \bfd W \to \bfd P \to \bfd V \to 0$ such that $\bfd P$ is also free and satisfies $\gd(\bfd V) = \gd (\bfd P)$. Then one has
\begin{equation*}
\hd_1(V) \leqslant \gd(W) \leqslant \gd(\bfd W) + 1 \leqslant \max \{ \hd_1(\bfd V), \, \gd(\bfd V) \} + 1 \leqslant \gd(\bfd V) + 1 = \gd(V)
\end{equation*}
since $\hd_1(\bfd V) = -1$. Putting the above two inequalities together we conclude that $\hd_1(V'') \leqslant \gd(V) = \gd(V'')$. By Statement (3) of Lemma \ref{properties of basic relative projective modules}, $V''$ is a basic relative projective module as claimed.

The conclusion follows after we show that $V'$ is relative projective. By the induction hypothesis, it suffices to show that $\bfd V'$ is relative projective. Since $V''$ is relative projective, and hence torsion free, by Statement (6) of Lemma \ref{torsion modules}, we get a short exact sequence $0 \to \bfd V' \to \bfd V \to \bfd V'' \to 0$. Note that $\bfd V$ is relative projective by the given condition, and so is $\bfd V''$ by Statement (4) of Lemma \ref{properties of basic relative projective modules}. The long exact sequence of homology groups tells us that $H_1(\bfd V') = 0$, so $\bfd V'$ is relative projective by Proposition \ref{properties of relative projective modules}.

(2): We use induction on $\gd(V)$. The conclusion holds if $\gd(V) = -1$; that is, $V = 0$. For $\gd(V) \geqslant 0$, consider the following commutative diagram:
\begin{equation*}
\xymatrix{
0 \ar[r] & V' \ar[r] \ar[d] & V \ar[r] \ar[d] & V/V' \ar[r] \ar[d]^{\delta} & 0\\
0 \ar[r] & \bfs V' \ar[r] & \bfs V \ar[r] & \bfs (V/V') \ar[r] & 0.
}
\end{equation*}
The kernel of $\delta$ is a torsion module, and is isomorphic to a submodule of $\bfd V'$ by the snake Lemma. But $V'$ is relative projective, so $\bfd V'$ is also relative projective, and hence is torsion free. This happens if and only if this kernel is 0. That is, $V/V'$ is torsion free.

By Statement (6) of Lemma \ref{torsion modules}, we get a short exact sequence $0 \to \bfd V' \to \bfd V \to \bfd (V/V') \to 0$. By Proposition \ref{properties of relative projective modules}, both $\bfd V'$ and $\bfd V$ are relative projective. Since $\gd(V) > \gd(\bfd V)$, by the induction hypothesis, $\bfd (V/V')$ is relative projective as well. The previous statement tells us that $V/V'$ is relative projective.
\end{proof}

\subsection{A proof of Theorem \ref{homological characterizations of relative projective modules}.}

Now we are ready to prove various homological characterizations of relative projective $\C$-modules.

\begin{proof}[A proof of Theorem \ref{homological characterizations of relative projective modules}]
The equivalence of the first three statements is implied by Proposition \ref{properties of relative projective modules}. Clearly, (2) implies (4). The argument showing that (4) implies (1) is completely the same as that of \cite[Proposition 4.4]{LY}. For the convenience of the reader, we include it here. Of course, we can assume that $s > 1$.

Take a projective resolution $P^{\bullet} \to 0$ and let $Z^i$ be the $i$-th cycle. Then we have
\begin{equation*}
0 = H_s(V) = H_1 (Z^{s-1}).
\end{equation*}
Consequently, $Z^{s-1}$ is a relative projective module. Applying the previous lemma to the short exact sequence
\begin{equation*}
0 \to Z^{s-1} \to P^{s-2} \to Z^{s-2} \to 0
\end{equation*}
one deduces that $Z^{s-2}$ is relative projective as well. The conclusion follows by recursion.
\end{proof}

A useful corollary is:

\begin{corollary} \label{ses of relative projective modules}
Let $0 \to U \to V \to W \to 0$ be a short exact sequence of $\C$-modules with finite generating degrees. Then if two terms in it are relative projective, so is the third term.
\end{corollary}

\begin{proof}
Consider the long exact sequence of homology groups and apply the above theorem.
\end{proof}

\subsection{Projective dimensions}

In this subsection we give an application of the previous theorem, classifying finitely generated $\C$-modules with finite projective dimension over commutative Noetherian rings. Recall that the finitistic dimension of a commutative Noetherian ring $k$, denoted by $\fdim k$, is defined to be
\begin{equation*}
\sup \{ \pd _{k} (T) \mid T \text{ is a finitely generated $k$-module and } \pd_{k} (T) < \infty \}.
\end{equation*}

\begin{proof}[A proof of Theorem \ref{projective dimension}]
The proof for $\FI$ described in \cite[Subsection 4.2]{LY} actually works for arbitrary $m \geqslant 1$ with small modifications. We give a sketch.

First, if $\pd_{k \C} (V)$ is finite, then it has a finite projective resolution. By the previous theorem, $V$ must be a relative projective module. Let $S$ be the set of objects $\bfn$ such that $(H_0(V))_{\bfn} \neq 0$. Then $V$ has a filtration such that each factor is a basic relative projective module $M(\bfn) \otimes _{k S_{\bfn}} W_{\bfn}$, $\bfn \in S$. As the proof of \cite[Lemma 4.7]{LY} (replacing $n$ in that proof by a maximal element in $S$), an induction on the size of $S$ asserts that
\begin{equation*}
\pd_{k \C} (V) \geqslant \pd_{k \C} (M(\bfn) \otimes _{k S_{\bfn}} W_{\bfn}) = \pd_{k S_{\bfn}} (W_{\bfn}) = \pd_{k} (W_{\bfn})
\end{equation*}
for $\bfn \in S$, where the two equalities can be shown as in \cite[Lemmas 4.5 and 4.6]{LY}. But from a standard homological argument one also has
\begin{equation*}
\pd_{k \C} (V) \leqslant \max \{ \pd_{k \C} (M(\bfn) \otimes _{k S_{\bfn}} W_n) \} _{\bfn \in S}.
\end{equation*}
Therefore, the description of $V$ holds if we assume that $\pd_{k \C} (V)$ is finite. The other direction is obvious.
\end{proof}

For semisimple rings or finite dimensional local algebras, one has:

\begin{corollary}
Let $k$ be a commutative Noetherian ring whose finitistic dimension is 0. Then a finitely generated $\C$-module has finite projective dimension if and only if it is projective.
\end{corollary}

\subsection{A proof of Theorem \ref{relative projective complexes}.}

In this subsection we prove Theorem \ref{relative projective complexes}. Let $k$ be a commutative Noetherian ring, and all $\C$-modules considered here are finitely generated. By Theorem \ref{noetherianity}, the category of finitely generated $\C$-modules is abelian.

Recall that $\td_i(V)$ for $i \in [m]$ and $\td(V)$ are defined in Definition \ref{torsion degree}. Moreover, by Lemma \ref{finite torsion degree}, these numbers are finite.

\begin{lemma} \label{shifted modules become torsion free}
Let $V$ be finitely generated $\C$-module over a commutative Noetherian ring and $i \in [m]$. Then $K_i \Sigma_i^{n_i} V = 0$ for $n_i \geqslant \td_i(V) + 1$. In particular, for $n > \td(V)$, the shifted module $\Sigma_1^n \ldots \Sigma_m^n V$ is torsion free.
\end{lemma}

\begin{proof}
From the proof of Lemma \ref{torsion modules} and Remark \ref{description of KV} one easily sees that $\Sigma_i^{n_i} K_iV = 0$ for $n_i > \td_i(V)$. As explained in the proof of Statement (4) of Lemma \ref{torsion modules}, one has $K_i \Sigma_i^{n_i} V \cong \Sigma_i^{n_i} K_iV = 0$.

If $n > \td(V)$, since $\td(V) \geqslant \td_i(V)$ for all $i \in [m]$, one has $K_i \Sigma_i^n V = 0$ for all $i \in [m]$. Therefore,
\begin{equation*}
K_i \Sigma_1^n \ldots \Sigma_m^n V \cong \Sigma_1^n \ldots \Sigma_{i-1}^n \Sigma_{i+1}^n \ldots \Sigma_m^n K_i \Sigma_i^n V = 0
\end{equation*}
by Statement (4) of Lemma \ref{basic properties of sigma}. Consequently, $\bfk \Sigma_1^n \ldots \Sigma_m^n V = 0$.
\end{proof}

However, $\bfs^n V$ might not be torsion free for any $n \in \N$. Here is an example.

\begin{example} \normalfont
Let $m = 2$ and $\bfn = (0, \, 0)$. Let $V$ be the $\C$-module $M(\bfn) / I_2 M(\bfn)$. Then for an object $\bft = (t_1, \, t_2)$, we have:
\begin{equation*}
V_{\bft} \cong \begin{cases}
k, & \text{ if } m_2 = 0;\\
0, & \text{ otherwise}.
\end{cases}
\end{equation*}
Clearly, $V$ is a torsion module. Moreover, one has $\Sigma_2 V = 0$ and $\Sigma_1 V \cong V$. Consequently, one has $\bfs V \cong V$, so $\bfs^n V$ can never be torsion free. However, since $\td(V) = 0$, one has $\Sigma_1 \Sigma_2 V = 0$.
\end{example}

\begin{proposition} \label{shifted modules become relative projective}
Let $V$ be a finitely generated $\C$-module over a commutative Noetherian ring $k$. Then for $n \gg 0$, the shifted module $\Sigma_1^n \ldots \Sigma_m^n V$ is relative projective.
\end{proposition}

\begin{proof}
Suppose that $V \neq 0$. By the previous lemma, $\Sigma_1^n \ldots \Sigma_m^n V$ is torsion free. Therefore, by Statement (1) of Lemma \ref{recursion}, it suffices to show that $\bfd \Sigma_1^n \ldots \Sigma_m^n V$ is relative projective. By Statement (5) of Lemma \ref{basic properties of sigma}, one has
\begin{equation*}
\bfd \Sigma_1^n \ldots \Sigma_m^n V \cong \Sigma_1^n \ldots \Sigma_m^n \bfd V.
\end{equation*}
Since $\gd(\bfd V) = \gd(V) - 1$ by Statement (3) of Lemma \ref{basic properties of sigma}, the induction hypothesis on generating degrees asserts that $\Sigma_1^n \ldots \Sigma_m^n \bfd V$ is relative projective. The conclusion follows by induction.
\end{proof}

\begin{proof}[A proof of Theorem \ref{relative projective complexes}]
The proof is almost the same as that of \cite[Theorem 4.14]{LY}, so we only give a brief explanation. For $n \gg 0$, applying the functor $\Sigma_1^n \ldots \Sigma_m^n$ to the short exact sequence $0 \to V_T \to V \to V_F \to 0$ one gets a short exact sequence $0 \to 0 \to F^0 \to F^0 \to 0$, where $\Sigma_1^n \ldots \Sigma_m^n V_T = 0$ by Lemma \ref{finite torsion degree} and $F^0 = \Sigma_1^n \ldots \Sigma_m^n V_F$ is a relative projective module by the previous proposition. The map $V \to F^0$ in the complex can be defined as the composite of the quotient $V \to V_F$ and the natural injection $V_F \to F^0$. Let $V^1$ be the cokernel of the map $V \to F^0$ and repeat the above procedure for it. Eventually one gets a complex $0 \to V \to F^0 \to F^1 \to \ldots$.

Note that $V^1 \cong \Sigma_1^n \ldots \Sigma_{m-1}^n \Sigma_m^{n-1} D_m V$. By (3) of Lemma \ref{basic properties of sigma}, $\gd(\bfd V) < \gd(V)$, so $\gd(D_mV) < \gd(V)$. Moreover, it also tells us that functors $\Sigma_i$, $i \in [m]$, do not increase the generating degree. Therefore, $\gd(V^1) < \gd(V)$. Consequently,
\begin{equation*}
\gd(V) \geqslant \gd(F^0) \geqslant \gd(V^1) \geqslant \gd(F^1) \geqslant \ldots
\end{equation*}
and
\begin{equation*}
\gd(V) > \gd(V^1) > \ldots.
\end{equation*}
Consequently, $l \leqslant \gd(V)$ and $\gd(F^j) \leqslant \gd(V) - j$.

From the definition of this complex one sees that the image of $V \to F^0$ is $V_F$ and the kernel is $V_T$. The map $F^i \to F^{i+1}$ is the composite of
\begin{equation*}
F^j \to V^{j+1} \to V_F^{j+1} \to F^{j+1},
\end{equation*}
whose image is $V^{j+1}_F$. As in the proof of \cite[Theorem 4.14]{LY}, one can show that the homology at $F^j$ is $V_T^{j+1}$, a finitely generated torsion module.

Suppose that $n_i \geqslant \td_i(H^j (F^{\bullet})) + 1$ for all $0 \leqslant j \leqslant l$ and $i \in [m]$. Applying $\Sigma_1^{n_1} \ldots \Sigma_m^{n_m}$ to the above complex, we get a shifted complex
\begin{equation*}
0 \to \Sigma_1^{n_1} \ldots \Sigma_m^{n_m} V \to \Sigma_1^{n_1} \ldots \Sigma_m^{n_m} F^0 \to \Sigma_1^{n_1} \ldots \Sigma_m^{n_m} F^ \to \ldots \to \Sigma_1^{n_1} \ldots \Sigma_m^{n_m} F^l \to 0.
\end{equation*}
By Lemma \ref{finite torsion degree}, all homology groups of $F^{\bullet}$ vanish under the shift. Therefore, the above shifted complex is exact. Now applying Corollary \ref{ses of relative projective modules} we deduce that $\Sigma_1^{n_1} \ldots \Sigma_m^{n_m} V$ is a relative projective module.
\end{proof}

\begin{remark} \label{optimal upper bound}
By Corollary \ref{ses of relative projective modules} and the construction of $F^{\bullet}$, $V$ is a relative projective module if and only if $F^{\bullet}$ is an exact sequence, if and only if all $H_j (F^{\bullet}) = 0$ for $0 \leqslant j \leqslant l$. With this observation, one can obtain the following stronger conclusion: $\Sigma_1^{n_1} \ldots \Sigma_m^{n_m} V$ is a relative projective module \textbf{if and only if}
\begin{equation*}
n_i > \max \{ \td_i(H_j(F^{\bullet})) \mid 0 \leqslant j \leqslant l \}.
\end{equation*}
\end{remark}

Define $N(V) = \max \{ \td(H^j (F^{\bullet})) \mid j \geqslant 0\}$, and for each $i \in [m]$, let $N_i(V) = \max \{ \td_i(H_j(F^{\bullet})) \mid j \geqslant 0 \}$ . These are finite numbers independent of the choice of particular relative projective complexes since different ones are quasi-isomorphic.

\section{Representation stability}

As a main application of the theorems previously established, in this section we deduce representation stability properties of finitely generated $\FI^m$-modules over fields, as we did for $\FI$-modules in \cite{GL2, L2, LR, LY}.

\subsection{Representation stability.} Representation stability of finitely generated $\FI$-modules was first observed in \cite{CEF}. We extend their result to finitely generated $\FI^m$-modules.

\begin{proof}[A proof of Theorem \ref{representation stability property}]
Clearly, if $V$ is representation stable, then for each object $\bfn$, $V_{\bfn}$ is finite dimensional by the third condition in Definition \ref{representation stability}. Moreover, the second condition in that definition tells us that $V$ is generated by its values on finitely many objects. Therefore, $V$ is finitely generated.

Conversely, if $V$ is finitely generated, then the second condition in Definition \ref{representation stability} must hold. To check the first and the third conditions, consider the complex
\begin{equation*}
F^{\bullet}: \quad 0 \to V \to F^0 \to F^1 \to \ldots \to F^l \to 0
\end{equation*}
where each $F^i$ is a finitely generated projective $\C$-module since $k$ is a field of characteristic 0. For an object $\bfn$ with $n_i \geqslant \max \{2\gd(V), \, N_i(V) + 1 \}$ for all $i \in [m]$, as explained in Remark \ref{optimal upper bound}, we get an exact sequence
\begin{equation*}
0 \to V_{\bfn} \to F^0_{\bfn} \to F^1_{\bfn} \to \ldots \to F^l_{\bfn} \to 0
\end{equation*}
since the values of all homology groups of $F^{\bullet}$ on $\bfn$ vanish, so the first condition in Definition \ref{representation stability} holds. Moreover, from Theorem \ref{relative projective complexes} we know that $\gd(F^j) \leqslant \gd(V)$. Now as Gan and the first author did for $\FI_G$ in \cite[Proof of Theorem 1.12]{GL2}, the reader can deduce the third condition in Definition \ref{representation stability} from the corresponding results of finitely generated projective modules, which was verified in \cite[Theorem 6.13]{Gad2}.\footnote{We use the condition that $n_i \geqslant 2\gd(V)$ for $i \in [m]$ to intrigue the conclusion of that theorem.}
\end{proof}

\subsection{Hilbert functions}
Throughout this subsection let $k$ be a filed of arbitrary characteristic, and $V$ be a finitely generated $\C$-module over $k$. The \emph{Hilbert function} of $V$ is defined by
\begin{equation*}
\Nm \to \N, \quad \bfn \mapsto \dim_{k} V_{\bfn}.
\end{equation*}

\begin{lemma}
Let $F$ be a finitely generated relative projective modules. Then there exist $m$ polynomials $P_i \in \mathbb{Q}[X]$ with degrees not exceeding $\gd(V)$ such that for objects $\bfn$ satisfying $n_i \geqslant \gd(V)$ for all $i \in [m]$, one has $\dim_k F_{\bfn} = P_1(n_1) \ldots P_m(n_m)$.
\end{lemma}

\begin{proof}
Since $F$ has a filtration by basic relative projective modules, it suffices to show that each filtration component satisfies the conclusion. That is, without loss of generality we can assume that $F$ is a basic relative projective module. Suppose that $F$ is generated by its value $F_{\bfl}$ on an object $\bfl$ with $\deg(\bfl) = \gd(F)$. Consequently, $F \cong M(\bfl) \otimes _{kS_{\bfl}} F_{\bfl}$.

Now let $\bfn$ be an object such that $n_i \geqslant \gd(F)$ for all $i \in [m]$. Clearly, $\bfn \geqslant \bfl$. Therefore,
\begin{equation*}
V_{\bfn} = k\C(\bfl, \bfn) \otimes _{kS_{\bfl}} V_{\bfl}
\end{equation*}
and hence
\begin{equation*}
\dim_k V_{\bfn} = \dim_k V_{\bfl} \cdot \prod_{i \in [m]} \frac{(n_i)!}{(n_i - l_i)!(l_i)!},
\end{equation*}
which implies the conclusion.
\end{proof}

We prove the polynomial growth property of $V$, extending \cite[Theorem B]{CEFN}.

\begin{proof}[A proof of Theorem \ref{polynomial growth}]
Again, consider the complex
\begin{equation*}
F^{\bullet}: \quad 0 \to V \to F^0 \to F^1 \to \ldots \to F^l \to 0.
\end{equation*}
For an object $\bfn$ with $n_i \geqslant \max \{\gd(V), \, N_i(V) + 1 \}$ for all $i \in [m]$, we get an exact sequence
\begin{equation*}
0 \to V_{\bfn} \to F^0_{\bfn} \to F^1_{\bfn} \to \ldots \to F^l_{\bfn} \to 0.
\end{equation*}
Moreover, by Theorem \ref{relative projective complexes} $\gd(F^i) \leqslant \gd(V)$. Now the reader only needs to check that those relative projective modules have the desired property, which is established in the previous lemma.
\end{proof}

\section{Questions and further remarks}

So far the reader can see that many representational and homological properties of $\FI$ extend to $\FI^m$ for arbitrary $m \geqslant 1$. However, compared to the rich results of representation theory of $\FI$, there are still many interesting parallel parts for $\FI^m$ deserving to be established. In this section we list a few questions.

\subsection{Regularity}

Church and Ellenberg have shown that $\reg(V) \leqslant \gd(V) + \hd_1(V) - 1$ for $\FI$-modules over any commutative ring. For arbitrary $m \geqslant 1$, Gan and the first author proved in \cite{GL3} that $\reg(V) < \infty$ for finitely generated $\FI^m$-modules over commutative Noetherian rings. However, for $m > 1$, there is no known upper bound for the regularity. Therefore, we wonder whether there exists an upper bound of $\reg(V)$ in terms of $\hd_i(V)$, $0 \leqslant i \leqslant m$, for any $\FI^m$-modules. If this is true, in particular, the category of $\FI^m$-modules satisfying $\hd_i(V) < \infty$, $0 \leqslant i \leqslant m$, is an abelian category containing all finitely generated $\FI^m$-modules, and one can work in this large category without worrying about the ground ring $k$.

\subsection{Coinduction functor}

The shift functors $\Sigma_i$, $i \in [m]$, are restriction functors, and hence have left and right adjoint functors. The left adjoint functors are inductions which can be easily defined, while the right adjoint functors, called \emph{coinduction functors}, are more delicate. In \cite{GL1} Gan and the first author defined coinduction functor for $\FI$-modules, explored its properties, and proved a few useful results. We hope to extend the nice properties of coinduction functor and their outcomes described in \cite{GL1, LR} to $\FI^m$-modules. In particular, when $k$ is a field of characteristic 0, is it still true that every finitely generated projective $\FI^m$-module is also injective?

\subsection{Local cohomology theory}

Motivated by the results described in Subsection 1.4 of this paper, we believe that there exists a local cohomology theory for $\FI^m$-modules, extending the work of Ramos and the first author for $\FI$-modules in \cite{LR}. In particular, the following statement might hold: Let $V$ be a finitely generated $\C$-module over a commutative Noetherian ring. Then $V$ is a relative projective module if and only if all local cohomology groups vanish. Moreover, the homology groups in the complex $F^{\bullet}$ coincide with local cohomology groups of $V$.

In a recent paper \cite{NSS}, Nagpal, Sam, and Snowden proved that for a finitely generated $\FI$-module over a commutative Noetherian ring, one has
\begin{equation*}
\reg(V) = \max \{ \td(H_i(V)) + i \mid i \geqslant 0\},
\end{equation*}
confirming a conjecture posted in \cite{LR}. We wonder whether there exists a similar equality for $\FI^m$.

\subsection{Applications in algebra, topology, and geometry}

Representation theory of $\FI$ has rich applications in representation stability theory. Indeed, the (co)homology groups of many sequences of mathematical objects have been found to be equipped with an $\FI$-module structure. Recently, Gadish described applications of finitely generated projective $\FI^m$-modules on a generalization of configuration spaces in \cite{Gad1, Gad2}. We believe that those results on finitely generated $\FI^m$-modules shall have more interesting applications.

\end{document}